\documentclass{article}

\usepackage{arxiv}
\usepackage[utf8]{inputenc} 
\usepackage[T1]{fontenc}    
\usepackage{hyperref}       
\usepackage{url}            
\usepackage{booktabs}       
\usepackage{amsfonts}       
\usepackage{nicefrac}       
\usepackage{microtype}      
\usepackage{lipsum}	
\usepackage{amsthm}
\usepackage{amsmath}

\title{Some remarks on sum of Euler's totient function}

\author{
  Es-said En-naoui\\
 \href{mailto:essaidennaoui1@gmail.com}{\tt essaidennaoui1@gmail.com}\\}
\newtheorem{theorem}{Theorem}
\newtheorem{corollary}[theorem]{Corollary}

\newtheorem{proposition}[theorem]{Proposition}

\newtheorem{remark}[theorem]{Remark}
\begin{document}
\maketitle

\begin{abstract}
Euler's totient function counts the positive integers up to a given integer $n$ that are relatively prime to $n$. The aim of this article is to give a result about the sum : 
$
\sum \limits_{\underset{p|k}{k=1}}^n \phi(k)
$ , for every prime number $p$ .
\end{abstract}

\section{Introduction}
Euler totient function $\phi$ is the function defined on the positive natural numbers $\mathbb{N}^*$ in the following way: if $n\in \mathbb{N}^*$, then we have\; :
$$
\phi (n)=\big\{x\in \mathbb{N}^*:1\leq x\leq n\;,\; \gcd (x,n)=1\big\}.
$$
 Clearly $\phi (p)=p-1$, for any prime number 
$p$ and, more generally, if $\alpha \in \mathbb{N}^*$, then $\phi (p^{\alpha})=p^{\alpha }- p^{\alpha -1}$. This follows from the fact that the only numbers which are not coprime with $p^{\alpha}$ are multiples of $p$ and there are $p^{\alpha -1}$ such multiples $x$ with $1\leq x\leq p^{\alpha}$.\\

It is well-known that $\phi$ is multiplicative, i.e., if $m$ and $n$ are 
coprime, then $\phi (mn) = \phi (m)\phi (n)$. If $n\geq 3$ and the prime 
decomposition of $n$ is 
$$
n = p_1^{\alpha _1}\ldots  p_s^{\alpha _s},
$$
then from what we have seen 
$$
\phi (n) = \prod _{i=1}^s(p_i^{\alpha _i}-p^{\alpha _i-1}) = 
n\prod _{i=1}^s(1-\frac{1}{p_i}).
$$
.

The following property is simple but very useful in this paper.

\begin{proposition}\label{prop1} If $p,m\in \mathbb{N}^*$, with $p$ prime, and $n=pm$, then $\phi (n)=(p-1)\phi (m)$, if $\gcd (p,m)=1$, and $\phi (n)=p\phi (m)$, if $\gcd (p,m)\neq 1$.
\end{proposition}

\begin{proof} If $\gcd (p,m)=1$, then we have
$$
\phi (n) =\phi (p)\phi (m) = (p-1)\phi (m).
$$
Now suppose that $\gcd (p,m)\neq 1$. We may write $m=p^{\alpha}m'$, with 
$\alpha \geq 1$ and $\gcd (p,m')=1$. Thus 
$$
\phi (n) = \phi (p^{\alpha +1})\phi (m') = p^{\alpha}(p-1)\phi (m').
$$
However 
$$
\phi (m) = p^{\alpha -1}(p-1)\phi (m')
$$
and so $$\phi (n)= p\phi (m)$$.
\end{proof}

\section{Calcul the sum  $\sum \limits_{\underset{p|k}{k=1}}^n \phi(k)$ }

Let $x$ be a real number , then $\lfloor x \rfloor$ is defined to be the unique integer $n$ satisfying $ n\leq x <n+1$.This integer  $\lfloor x \rfloor$  is called the floor of $x$. Then we have rhis results :
\begin{corollary}\label{cor1}
 Let $m$ be an integer, and let $x$ be a real number. Then,\;$ m \leq x$ holds if and only if \;$m \leq \lfloor x \rfloor$ holds.
\end{corollary}
\begin{proof}
Recall that $\lfloor x \rfloor$ is the unique integer $n$ satisfying $n \leq x < n + 1$. \\
Thus, $\lfloor x \rfloor$ is an integer satisfying $\lfloor x \rfloor \leq x < \lfloor x \rfloor+1$ .\\
If $m \leq x$ holds, then $m \leq \lfloor x \rfloor$ holds . Conversely, if $m \leq \lfloor x \rfloor$ holds, then $m \leq x$ (because $m \leq \lfloor x \rfloor \leq x$). Combining these two implications, we conclude that $m \leq x$ holds if and only if $m \leq \lfloor x \rfloor$ holds. Corollary 
$ \ref{cor1} $ is thus proven.
\end{proof}

\begin{proposition}\label{prop2}
Let $x\in\mathbb{R}$ and $m\in\mathbb{N}^*$.\;Then , 
$\big\lfloor  \frac{\lfloor x \rfloor}{m}   \big\rfloor
=\big\lfloor \frac{x}{m}\big\rfloor $. 
\end{proposition}
\begin{proof}
By the definition of $\lfloor x \rfloor $,\; that $\big\lfloor \frac{x}{m} \big\rfloor$ is the unique integer $n$ satisfying $n \leq \frac{x}{m} < n + 1$.Thus, $\big\lfloor \frac{x}{m} \big\rfloor$ is an integer satisfying $\big\lfloor \frac{x}{m} \big\rfloor \leq \frac{x}{m} < \big\lfloor \frac{x}{m} \big\rfloor+1$ .\\
But $m\in\mathbb{N}^*$ and thus $m\geq 1 >0 $ . Hence, we can multiply the inequality $\big\lfloor \frac{x}{m} \big\rfloor \leq \frac{x}{m}$ by m. We thus obtain $m\big\lfloor \frac{x}{m} \big\rfloor \leq x$.
But  the Corollary $ \ref{cor1} $ (applied to $m\big\lfloor \frac{x}{m}\big\rfloor$ instead of $m$) shows that $m\big\lfloor \frac{x}{m} \big\rfloor \leq x$ holds if and only $m\big\lfloor \frac{x}{m} \big\rfloor \leq \lfloor x\rfloor$ holds (since $m\big\lfloor \frac{x}{m}\big\rfloor$ is an integer). Thus,
$m\big\lfloor \frac{x}{m} \big\rfloor \leq \lfloor x\rfloor$ holds (since $m\big\lfloor \frac{x}{m} \big\rfloor \leq  x$ holds ).Dividing this inequality by $m$, we
obtain $\big\lfloor \frac{x}{m} \big\rfloor \leq \frac{\lfloor x\rfloor}{m}$. We can now divide the inequality $\lfloor x\rfloor\leq x$ by $m$ (since $m > 0$).We thus obtain
$\frac{\lfloor x\rfloor}{m} \leq \frac{x}{m}$.\\
Hence , $\frac{\lfloor x \rfloor}{m}\leq \frac{x}{m}<
\big\lfloor\frac{x}{m} \big\rfloor+1  $.So we have 
$\big\lfloor\frac{x}{m}\big\rfloor\leq \frac{\lfloor x \rfloor}{m}<\big\lfloor\frac{x}{m}\big\rfloor+1$.In other words,$\big\lfloor\frac{x}{m}\big\rfloor$ is an integer $n$ satisfying 
$n \leq \frac{\lfloor x\rfloor}{m}<n+1$.But 
$\big\lfloor \frac{\lfloor x\rfloor}{m} \big\rfloor $ is the unique integer $n$ satisfying $n \leq \frac{\lfloor x\rfloor} {m}<n+1$ (because this is
how $\big\lfloor \frac{\lfloor x\rfloor}{m} \big\rfloor $ is defined).Hence, if $n$ is any integer satisfying $n \leq \frac{\lfloor x\rfloor}{m}<n+1$, then $n=\big\lfloor \frac{\lfloor x\rfloor}{m} \big\rfloor$.We can apply this to $n =\big\lfloor\frac{x}{m} \big\rfloor$ (since $\big\lfloor\frac{x}{m} \big\rfloor$ is an integer $n$ satisfying $n\leq \big\lfloor \frac{x}{m}\big\rfloor<n+1$), and thus obtain $\big\lfloor  \frac{\lfloor x \rfloor}{m}   \big\rfloor =\big\lfloor \frac{x}{m}\big\rfloor $.Proposition $  \ref{prop2} $ is proven .
\end{proof}

\begin{proposition}\label{pro3} 
Let $n$ a positive integer and $p$ a number prime, put $f(n)=\big\lfloor\frac{n}{p}\big\rfloor$ , Then we have :\\
\begin{equation}
\forall k\in\mathbb{N^*}\;\;,\;\;\underbrace{fo \cdots of}_{k\;\; times}(n)=\bigg\lfloor\frac{n}{p^k}\bigg\rfloor 
\end{equation}  
\end{proposition}
\begin{proof}
we prove by induction on k that , then When $k = 1$, it is clear that $f(n)=\big\lfloor\frac{n}{p}\big\rfloor$.\\
Assume that \;\;
$\underbrace{fo \cdots of}_{k\;\; times}(n)=\bigg\lfloor\frac{n}{p^k}\bigg\rfloor$. Then we have :
\begin{align*}
\underbrace{fo \cdots of}_{k+1\; times}(n)=
fo\underbrace{fo \cdots of}_{k\; times}(n)
=f\bigg( \frac{\big\lfloor\frac{n}{p^k}\big\rfloor}{p}  \bigg)
\end{align*}
We let $x=\big\lfloor\frac{n}{p^k}\big\rfloor $ to substitute into proposition $\ \ref{prop2} $, then We will find : 
\begin{equation*}
\underbrace{fo \cdots of}_{k+1\; times}(n)=
\bigg\lfloor\frac{n}{p^{k+1}}\bigg\rfloor 
\end{equation*}
\end{proof}

Now after we give some property useful in this paper ,then , let n be a positive integer ,and  put :$\Psi(n)=\sum \limits_{k=1}^n \phi(k)$\;, So for every prime number  $p$ we put : 
\[
	\Upsilon(n,p)=\sum \limits_{\underset{(k,p)=1}{k=1}}^n \phi(k)\;\;\;and \;\;\; \Delta(n,p)=\sum \limits_{\underset{p|k}{k=1}}^n \phi(k)
\]
 Then we have \;:
\begin{equation}\label{equa1}
\Psi(n)=\Upsilon(n,p)+\Delta(n,p)
\end{equation}

\begin{theorem} \label{thm3a}For every positive integer $n\geq2$, and a prime number $p$ ,  we have :\\
\begin{equation}
\Delta(n,p)=(p-1)\sum_{\alpha=1}^{\big\lfloor\log_p(n)\big\rfloor} \Psi\bigg(\bigg\lfloor\frac{n}{p^\alpha}\bigg\rfloor\bigg)
\end{equation}
\end{theorem}
\begin{proof} we have :
\begin{align*} 
\sum \limits_{\underset{p|k}{k=1}}^n \phi(k)
& =\phi(p) + \phi(2p) + \cdots + \phi(jp) \;\;,\;with\;\; j\leq \big\lfloor\frac{n}{p}\big\rfloor\\
& =\sum_{j=1}^{\big\lfloor\frac{n}{p}\big\rfloor} \phi(jp) = \sum \limits_{\underset{p|j}{j=1}}^{\big\lfloor\frac{n}{p}\big\rfloor} \phi(jp)+\sum \limits_{\underset{(p,j)=1}{j=1}}^{\big\lfloor\frac{n}{p}\big\rfloor} \phi(jp) 
\end{align*}

From the proposition $\ref{prop1}$ , we have $\phi(jp)=(p-1)\phi(j)$ if $(j,p)=1$ and $\phi(jp)=p\phi(j)$ if $(j,p)>1$ then  : 
\begin{align*} 
\sum \limits_{\underset{p|k}{k=1}}^n \phi(k)
& = p\sum \limits_{\underset{p|j}{j=1}}^{\big\lfloor\frac{n}{p}\big\rfloor} \phi(j)+(p-1)\sum\limits_{\underset{(p,j)=1}{j=1}}^{\big\lfloor\frac{n}{p}\big\rfloor} \phi(j) \\
& =p\Delta\big(\big\lfloor\frac{n}{p}\big\rfloor,p\big)+(p-1)\Upsilon\big(\big\lfloor\frac{n}{p}\big\rfloor,p\big)\\
& =p\Delta\big(\big\lfloor\frac{n}{p}\big\rfloor,p\big)+p\Upsilon\big(\big\lfloor\frac{n}{p}\big\rfloor,p\big)-\Upsilon\big(\big\lfloor\frac{n}{p}\big\rfloor,p\big)\\
&=p\bigg( \underbrace{\Delta\big(\big\lfloor\frac{n}{p}\big\rfloor,p\big)+\Upsilon\big(\big\lfloor\frac{n}{p}\big\rfloor,p\big)}_{=\Psi\big(\big\lfloor\frac{n}{p}\big\rfloor\big) \;\; by\;\;\ref{equa1}} \bigg)-\Upsilon(\big\lfloor\frac{n}{p}\big\rfloor,p\big)\\
&=p\Psi\big(\big\lfloor\frac{n}{p}\big\rfloor\big)-\Upsilon\big(\big\lfloor\frac{n}{p}\big\rfloor,p\big)
\end{align*}

Since 
$\Upsilon\big(\big\lfloor\frac{n}{p}\big\rfloor,p\big)=
\Psi\big(\big\lfloor\frac{n}{p}\big\rfloor\big) -
\Delta\big(\big\lfloor\frac{n}{p}\big\rfloor,p\big)
 $  by equality $ \ref{equa1} $ , Then  :
\begin{align*} 
\Delta(n,p)&=p\Psi\big(\big\lfloor\frac{n}{p}\big\rfloor\big)-\bigg(\Psi\big(\big\lfloor\frac{n}{p}\big\rfloor\big)-\Delta\big(\big\lfloor\frac{n}{p}\big\rfloor,p\big)\bigg)
\end{align*}
so that : 
\begin{equation}\label{equa4}
\Delta(n,p)=(p-1)\Psi\big(\big\lfloor\frac{n}{p}\big\rfloor\big)+\Delta\big(\big\lfloor\frac{n}{p}\big\rfloor,p\big)
\end{equation} 

In the equality $ \ref{equa4} $,  Substituting $n$ by $\big\lfloor\frac{n}{p} \big\rfloor $ , and by using the proposition $ \ref{pro3}  $ we have  :
\begin{equation}
\Delta\big(\big\lfloor\frac{n}{p}\big\rfloor,p\big)=(p-1)\Psi\bigg(\big\lfloor\frac{n}{p^2}\big\rfloor\bigg)+\Delta\bigg(\big\lfloor\frac{n}{p^2}\big\rfloor,p\bigg)
\end{equation}
Now , Substituting this  $\Delta\big(\big\lfloor\frac{n}{p}\big\rfloor,p\big)$ into $4$ we find that:
\begin{equation}
\Delta(n,p)=(p-1)\Psi\big(\big\lfloor\frac{n}{p}\big\rfloor\big)+(p-1)\Psi\big(\big\lfloor\frac{n}{p^2}\big\rfloor\big)+\Delta\big(\big\lfloor\frac{n}{p^2}\big\rfloor,p\big)
\end{equation}
In the same way if we repeat this operation up to an integer $\alpha$ such that $\big\lfloor\frac{n}{p^{\alpha}}\big\rfloor\leq 1$ finds this equality :
\begin{align*}
\Delta(n)
& =(p-1)\Psi\big(\big\lfloor\frac{n}{p}\big\rfloor\big)+(p-1)\Psi\big(\big\lfloor\frac{n}{p^2}\big\rfloor\big)+(p-1)\Psi\big(\big\lfloor\frac{n}{p^3}\big\rfloor\big)+ \cdots +
\Psi\big(\big\lfloor\frac{n}{p^\alpha}\big\rfloor\big)\\
&=(p-1)\sum_{\alpha=1}^{\big\lfloor\log_p(n)\big\rfloor} \Psi\big(\big\lfloor\frac{n}{p^\alpha}\big\rfloor\bigg)
\end{align*}

\end{proof}

\begin{remark} For every prime number $p$ we have : 
$$\frac{\Delta(pn,p)}{p-1}=\Psi(n)+\Delta(n,p).$$
So for $p=2$ , we have :
\begin{equation*}
\Delta(2n,2)=\Psi(n)+\Delta(n,2).
\end{equation*}
\end{remark}

\bibliographystyle{unsrt}  


\end{document}